\documentclass{article}
\usepackage{amsmath,amsthm,amsopn,amstext,amscd,amsfonts,amssymb,verbatim}
\usepackage{color}
\usepackage[numbers]{natbib}

\def\tr{\mathop{\rm tr}\nolimits}

\def\G{\mathop{\rm G}\nolimits}
\def\ch{\mathop{\rm ch}\nolimits}
\def\diag{\mathop{\rm diag}\nolimits}

\def\etr{\mathop{\rm etr}\nolimits}

\def\re{\mathop{\rm Re}\nolimits}

\newcommand {\boldgreektext}[1] {\boldmath
             \(#1\)\unboldmath}
\newcommand {\boldgreek}[1]
             {\mbox{\boldgreektext{#1}}
            }

\renewenvironment{abstract}
                 {\vspace{6pt}
                  \begin{center}
                  \begin{minipage}{5in}
                  \centerline{\textbf{Abstract}}
                  \noindent\ignorespaces
                 }
                 {\end{minipage}\end{center}}
\newtheorem{thm}{\textbf{Theorem}}[section]
\newtheorem{cor}{\textbf{Corollary}}[section]
\newtheorem{lem}{\textbf{Lemma}}[section]

\theoremstyle{definition}
\newtheorem{defn}{\textbf{Definition}}[section]

\title{\Large \textbf{Singular matrix variate Birnbaum-Saunders distribution under elliptical models}}
\author{
  \textbf{Jos\'e A. D\'{\i}az-Garc\'{\i}a} \thanks{Corresponding author\newline
   {\bf Key words.}  Singular matrix variate distributions, singular elliptical distributions, Birnbaum-Saunders
        distribution, Hausdorff measure.\newline
    2000 Mathematical Subject Classification. 62E15; 60E05; 15A23; 15A09; 15B52}\\
  {\normalsize Universidad Aut\'onoma de Chihuahua} \\
  {\normalsize Facultad de Zootecnia y Ecolog\'{\i}a} \\
  {\normalsize Perif\'erico Francisco R. Almada Km 1, Zootecnia} \\
  {\normalsize 33820 Chihuahua, Chihuahua, M\'exico}\\
  {\normalsize E-mail: jadiaz@uach.mx}\\
  \textbf{Francisco J. Caro-Lopera}\\
  {\normalsize Departament of Basic Sciences} \\
  {\normalsize Universidad de Medell\'{\i}n} \\
  {\normalsize Medell\'{\i}n, Colombia} \\
  {\normalsize E-mail: fjcaro@udem.edu.co} \\[2ex]
}
\date{}
\begin{document}
\maketitle

\begin{abstract}
This work sets the matrix variate Birnbaum-Saunders theory in the context of singular distributions and
elliptical models. The so termed singular matrix variate generalised Birnbaum-Saunders distribution is
obtained with respect the Hausdorff measure. Several basic properties and particular cases of this
distribution are also derived.
\end{abstract}

\section{Introduction}\label{sec:1}

The univariate \emph{Birnbaum-Saunders distribution}, introduced by \citet{bs:69}, has promoted a
considerable research during the last 50 years. At first, the distribution was motivated as a lifetime
model for fatigue failure caused under cyclic loading, under the assumption that the failure is due to
the development and growth of a dominant crack. A more general derivation was also provided by
\citet{d:85} in a context of a biological model.

However, by passing the decades, the advances in more complex scenarios were so slow. In fact, the
literature of non singular matrix variate Birnbaum-Saunders is such small that in the recent review by
\citet{bk:18}, only 1 of the 281 cited papers belonged to the matrix case (\citet{cl:12}). In a
discussion about the addressed review, the number has increased to only 3 more new works, written by the
same group of authors, \citet{dgcl:18}.

Now, when the research moves to the singular case, the problems are greater, because the distributions do
not exist with respect the Lebesgue measure, see \citet{k:68}. For a unified approach, summarising a
number of singular distributions in a wider spectra, see \citet{k:68}, \citet{uh:94}, \citet{r:05},
\citet{dggjm:97,dggg:05a,dggg:05b,dggj:05,dggj:06a,dggj:06b}, \citet{z:07}, \citet{bo:08} and the
references therein.

Finally, the evolution of matrix variate distributions is usually given by exploring the Gaussian case
and then providing a generalisation under families of distributions. The generalised theory also based on
singular distribution allows the desirable unified approach for this distributional challenge.

For a contextualisation of the problem, we provide some highlights of the distribution.

The original univariate random variable Birnbaum-Saunders was obtained as a function of a normal random
variable. Furthermore, if $Z \sim \mathcal{N}(0,1)$ then $T$ is called a Birnbaum-Saunders random
variable, where
\begin{equation}\label{den}
    T = \beta\left (\frac{\alpha}{2}Z +\sqrt{\left(\frac{\alpha}{2}Z \right)^{2}+1} \right )^{2},
\end{equation}
We shall denote this fact as $T \sim \mathcal{BS}(\alpha, \beta)$, where $\alpha > 0$ is the shape
parameter, and $\beta > 0$ is both scale parameter an the median value of the distribution. Thus, the
inverse relation establishes that
 if $T \sim \mathcal{BS}(\alpha, \beta)$, then
\begin{equation}\label{ne}
    Z = \frac{1}{\alpha}\left (\sqrt{\frac{T}{\beta}} - \sqrt{\frac{\beta}{T}}\right ) \sim \mathcal{N}(0,1)
\end{equation}
\citet{dg:05,dge:06} proposed a generalisation of the Birnbaum-Saunders distribution, replacing the
normal distribution hypothesis in (\ref{ne}) by a \emph{symmetric distribution}, i.e. they assume that $Z
\sim \mathcal{E}(0,1,h)$. We recall that the density function of $Z \sim \mathcal{E}(0,1,h)$ is defined
as $f_{Z}(z) =  h(z^{2})$, for $z \in \Re$. Therefore, (\ref{den}) defines the \emph{generalised
Birnbaum-Saunders distribution}, which shall be denoted by $T \sim \mathcal{GBS}(\alpha, \beta; h)$.

From \citet{dg:05,dge:06} if $T \sim \mathcal{GBS}(\alpha, \beta,h)$, then
\begin{equation}\label{bs}
    dF_{T}(t)  =  \frac{t^{-3/2}\left(t + \beta\right)}{2\alpha\sqrt{\beta}}\
  h\left[\frac{1}{\alpha^{2}}\left(  \frac{t}{\beta }+\frac{\beta}{t}-2\right)\right] dt, \quad t > 0.
\end{equation}
Alternatively, let $V = \sqrt{T}$, with $dt = 2vdv$, then under a symmetric distribution, (\ref{ne}) can
be rewritten as
\begin{equation}\label{nem}
    Z = \frac{1}{\alpha}\left (\frac{V}{\sqrt{\beta}} - \frac{\sqrt{\beta}}{V}\right ),
\end{equation}
and its density is given by
\begin{equation}\label{bsm}
    dF_{V}(v)  =  \displaystyle \frac{\left (1 + \beta v^{-2} \right )}{\alpha \sqrt{\beta}}\
  h\left[ \frac{1}{\alpha^{2}}\left( \frac{v^{2}}{\beta }+\frac{\beta}{v^{2}}-2\right)\right]dv, \quad v > 0,
\end{equation}
which shall be termed \emph{square root generalised Birnbaum-Saunders distribution}.

The study of the Birnbaum-Saunders distribution has been very profuse in the univariate case (basic
properties and estimation), and few works exist in the extension to multivariate (vector) case under
elliptical models. As we quoted \citet{bk:18} make a detailed compilation of this distribution in the
last five decades. Especially, the results in the specialised literature for these multivariate
generalisations, have been proposed defining the vector and matrix variate cases, element-to-element. For
the vectorial case we mention \citet{dd:06}, \citet{dd:07} and for the matrix case we have \citet{cl:12},
\citet{slcc:15}, \citet{cldg:16}.  Recently, \citet{dgcl:19b} closed the problem for a non singular
matrix variate version in terms of a matrix transformation open a perspective to the solution in the
singular case.

Thus in this paper we obtain a singular version of this matrix transformation and then we find the
corresponding singular matrix variate generalised Birnbaum-Saunders distribution with respect to the
Hausdorff measure (see \citet[Section 19]{b:86} and \cite{dggjm:97}). For this task we need
consider the following aspects.
\begin{itemize}
  \item Given a singular random matrix, if a density function with respect certain measure is obtained,
    we must know that both the density and the measure are non unique. However, no matters what density
    and corresponding measure is under consideration, they provide the same probability results,
    inference, etc., see \citet{k:68} and \citet{r:05}.
  \item Another very important aspect warns about the combination of results under different densities
    and measures.
  \item In terms of the previous point, this is the source of explanation for a number of wrong applications
    of the published results, see \citet{dg:07}.
  \item As a consequence, when alternative approaches of literature are used in the setting of the present
    work, it should be noted that inconsistent algebraic, probabilistic and conceptual results are obtained.
    Furthermore, in those papers the approach followed in this work is validated by different authors, see
    \citet{z:07}, \citet{dg:07}, and \citet{bo:08}, among others.
  \item A consistent presentation on the theory of singular random matrix and vector distributions can be
    found in \citet{k:68}, \citet{uh:94}, \citet{r:05}, \citet{dggj:97,dggj:05,dggj:06a, dggj:06b}, \citet{dggjm:97},
    \citet{dggg:05a,dggg:05b},
    among others.
\end{itemize}

The above discussion is placed in the paper in two parts. Some preliminary results and new Jacobians are
studied in Section \ref{sec:2}. Then, Section \ref{sec:3} proposes the main result of the article.
Finally, certain special cases and some basic properties are derived.

\section{Preliminary results}\label{sec:2}

Some preliminary results about the singular matrix variate elliptical distribution are summarised below.

Two Jacobians for matrix transformations with respect to Hausdorff measure are computed. First, some
results and notations about the required matrix algebra are considered, see \citet{dggg:05b},
\citet{dggj:05}, \citet{r:05} and \citet{mh:05}.

\subsection{Notation}

Let $\mathcal{L}_{m,n}(q)$ be the linear space of all $n \times m$ real matrices of rank $q \leq
\min(n,m)$ and ${\mathcal L}_{m,n}^{+}(q)$ be the linear space of all $N \times m$ real matrices of rank
$q \leq \min(N,m)$ with $q$ distinct singular values. The set of matrices $\mathbf{H}_{1} \in {\mathcal
L}_{m,n}$ such that $\mathbf{H}'_{1}\mathbf{H}_{1} = \mathbf{I}_{m}$ is a manifold denoted ${\mathcal
V}_{m,n}$, called Stiefel manifold. In particular, ${\mathcal V}_{m,m}$ is the group of orthogonal
matrices ${\mathcal O}(m)$. Denote by ${\mathcal S}_{m}$, the homogeneous space of $m \times m$ positive
definite symmetric matrices; $\mathcal{S}_{m}^{+}(q)$, the $(mq - q(q - 1)/2)$-dimensional manifold of
rank $q$ positive semidefinite $m \times m$ symmetric matrices with $q$ distinct positive eigenvalues.
For all matrix $\mathbf{A} \in \mathcal{L}_{m,n}(q)$ exist $\mathbf{A}^{+} \in \mathcal{L}_{n,m}(q)$
which is termed \emph{Moore-Penrose inverse}. Similarly, for all matrix $\mathbf{A} \in
\mathcal{L}_{m,n}(q)$ exist $\mathbf{A}^{-} \in \mathcal{L}_{m,n}(r)$, $r \geq q$ which is termed
\emph{conditional (or generalised) inverse} is such that $\mathbf{AA}^{-}\mathbf{A} = \mathbf{A}$. The
\emph{eigenvalues} of $\mathbf{A} \in \mathcal{L}_{m,m}(q)$ are the roots of the equation
$|\mathbf{A}-\lambda \mathbf{I}_{m}| = 0$. The $i$-th eigenvalue of $\mathbf{A}$ shall be denoted as
$\ch_{i}(\mathbf{\mathbf{A}})$. Given  $\mathbf{A} \in \mathcal{S}_{m}^{+}(q)$, there exist
$\mathbf{A}^{1/2} \in \mathcal{S}_{m}^{+}(q)$ such that $\mathbf{A} = \left(\mathbf{A}^{1/2}
\right)^{2}$, which is termed \emph{positive (definite) semi-definite root matrix}.

\subsection{Matrix variate distribution.}

\begin{defn} \label{def1}
It is said that $\mathbf{Y} \in \mathcal{L}_{m,n}(q)$, $q = \min(r,s)$, has a \emph{singular matrix variate
elliptically contoured distribution} if its density $dF_{\mathbf{Y}}(\mathbf{Y})$ is given by:
$$
  =\frac{1}{\left(\displaystyle\prod_{i=1}^{r} \ch_{i}(\mathbf{\Sigma})^{s/2}\right) \
  \left(\displaystyle\prod_{j=1}^{s} \ch_{j}(\mathbf{\Theta})^{r/2}\right)} h\left(\tr \mathbf{\Sigma}^{-} (\mathbf{Y} -
  \boldgreek{\mu})' \mathbf{\Theta}^{-} (\mathbf{Y} -   \boldgreek{\mu})\right)(d\mathbf{Y})
$$
where  $\boldsymbol{\mu} \in \Re^{n\times m}$, $\mathbf{\Sigma} \in \mathcal{S}_{m}^{+}(r)$, $
\mathbf{\Theta} \in \mathcal{S}_{n}^{+}(s)$, and $(d\mathbf{Y})$ is the Hausdorff measure. The function
$h: \Re \rightarrow [0,\infty)$ is termed the generator function and satisfies $\int_{0}^\infty
u^{rs-1}h(u^2)du < \infty$. Such a distribution is denoted by $\mathbf{Y}\sim \mathcal{E}_{n\times
m}^{r,s}(\boldsymbol{\mu},\mathbf{\Theta} \otimes \mathbf{\Sigma}, h)$, omitting the supra-index when $r
= m$ and $s = n$, see \citet{dggg:05b}. Observe that if $\mathbf{Y} \in \mathcal{L}_{m,n}^{+}(q)$, then
there exist $\mathbf{V}_{1} \in {\mathcal V}_{q,n}$, $\mathbf{W}'_{1} \in {\mathcal V}_{q,m}$ and
$\mathbf{L}= \diag(l_{1}, \dots,l_{q})$, $l_{1}> \cdots > l_{q}>0$, such that $\mathbf{Y} =
\mathbf{V}_{1}\mathbf{LW}'_{1}$, is the nonsingular part of the singular value decomnposition (SVD),
\citet[p. 42, 1973]{r:05}. Then, the Hausdorff measure $(d\mathbf{Y})$ can be explicitly defined as follows
\begin{equation}\label{h1}
    (d\mathbf{Y}) = 2^{-q} \prod_{i=1}^{q}l_{i}^{n + m - 2q} \prod_{i < j}^{q}(l_{i}^{2} - l_{j}^{2})
          (\mathbf{V}'_{1}d\mathbf{V}_{1})(\mathbf{W}'_{1}d\mathbf{W}_{1})\bigwedge_{i=1}^{q}dl_{i},
\end{equation}
but quoting that the density and the measure defined in this way are not unique, see \citet{k:68} and
\citet{dggg:05b}.
\end{defn}

When $\boldsymbol{\mu}=\mathbf{0}_{n\times m}$, $\mathbf{\Sigma}= \mathbf{I}_{m}$ and $ \mathbf{\Theta} =
\mathbf{I}_{n}$, such distribution is termed \emph{matrix variate symmetric distribution} and shall
be denoted as $\mathbf{Y} \sim \mathcal{E}_{n \times m}(\mathbf{0}, \mathbf{I}_{nm}, h)$.

This class of matrix variate distributions includes\emph{ normal, contaminated normal, Pearson type II
and VI, Kotz, logistic, power exponential}, and so on; these distributions have tails that are weighted
more or less, and/or they have greater or smaller degree of kurtosis than the normal distribution.

In addition, note that if $\mathbf{Y}\sim \mathcal{E}_{n\times m}^{r,s}(\boldsymbol{\mu},\mathbf{\Theta}
\otimes \mathbf{\Sigma}, h)$, and $\mathbf{A} \in \mathcal{L}_{n,a}(k)$ and $\mathbf{B} \in
\mathcal{L}_{m,b}(t)$. Then, $\mathbf{AYB}'\sim \mathcal{E}_{a\times b}^{\alpha,\beta}(\boldsymbol{A\mu
B}',\mathbf{A}'\mathbf{\Theta A} \otimes \mathbf{B'}\mathbf{\Sigma B}, h)$, where $\alpha$ and $\beta$
are the ranks of $\mathbf{A}'\mathbf{\Theta A}$ and $\mathbf{B'}\mathbf{\Sigma B}$, respectively; where
$\alpha \leq \min(r,k)$ and $\beta \leq \min(s,t)$.

\subsection{Jacobians}

\begin{lem}\label{lem1}
Let $\mathbf{U}$ and $\mathbf{W} \in \mathcal{L}_{m,n}(p)$, such that
\begin{equation}\label{WW}
    \mathbf{U} = \mathbf{W} - \mathbf{W}^{'+}.
\end{equation}
Then
\begin{equation}\label{dUU1}
  (d\mathbf{U}) =
        \left\{
              \begin{array}{l}
                 \displaystyle\prod_{i=1}^{p}\left(1 - d_{i}^{-2}\right)^{n+m-2p} \left(1+d_{i}^{-2}\right)
                 \prod_{i<j}^{p}\left(1- d_{i}^{-2}d_{j}^{-2}\right)(d\mathbf{W})\\
                 \displaystyle\prod_{i=1}^{p}d_{i}^{-2(n+m-p)}\left(d_{i}^{2}-1\right)^{n+m-2p} \left(1+d_{i}^{2}\right)
                 \prod_{i<j}^{p}\left(d_{i}^{2}d_{j}^{2}-1\right)(d\mathbf{W}),
              \end{array}
        \right.
\end{equation}
where $d_{i}^{2}= \ch_{i}(\mathbf{W}'\mathbf{W})$, $i = 1,2,\dots,p$, $d_{1} > d_{2} > \cdots
> d_{p} > 0$.
\end{lem}
\begin{proof}
Let $\mathbf{W} = \mathbf{H}_{1}\mathbf{DQ}_{1}'$ the non-singular part of the singular value
factorisation of $\mathbf{W}$, where $\mathbf{H}_{1} \in \mathcal{V}_{p,m}$, $\mathbf{D} = \diag(d_{1},
\dots, d_{p})$, $d_{1}> \cdots > d_{p} > 0$ and $\mathbf{Q}_{1} \in \mathcal{V}_{p,n}$, with
$d_{i}^{2}=\ch_{i}(\mathbf{W}'\mathbf{W})$, see \citet[Theorem A9.10, p. 593]{mh:05}. By \citet[Problem
28e, pp. 76-77]{r:05} is know that $\mathbf{W}^{+} = \mathbf{Q}_{1}\mathbf{D}^{-1}\mathbf{H}'_{1}$. Then
from (\ref{WW})
\begin{eqnarray*}
  \mathbf{U}&=& \ \mathbf{H}_{1}\mathbf{D}\mathbf{Q}'_{1}- \left(\mathbf{Q}_{1}\mathbf{D}^{-1}\mathbf{H}'_{1}\right)' \\
   &=& \mathbf{H}_{1}\left(\mathbf{D}-\mathbf{D}^{-1}\right)\mathbf{Q}'_{1}.
\end{eqnarray*}
By \citet{dggj:05}, if we take $g(\alpha_{i}) = d_{i}-d_{i}^{-1}$ then we obtain
\begin{equation}\label{sv0}
    (d\mathbf{U}) = \prod_{i=1}^{p}\left( \frac{d_{i}-d_{i}^{-1}}{d_{i}}\right)^{n+m-2p}  \prod_{i<j}^{p}
  \frac{\left(d_{i}-d_{i}^{-1}\right)^{2}-\left(d_{j}-d_{j}^{-1}\right)^{2}}{d_{i}^{2}-d_{j}^{2}}
$$
$$
  \hspace{5cm}\times \ \prod_{i=1}^{p} \frac{d\left(d_{i}-d_{i}^{-1}\right)}{dd_{i}} (d\mathbf{W}).
\end{equation}
Observe that
\begin{equation}\label{sv1}
     \prod_{i=1}^{p} \frac{d\left(d_{i}-d_{i}^{-1}\right)}{dd_{i}} = \left\{
        \begin{array}{l}
          \displaystyle\prod_{i=1}^{p}\left(1+d_{i}^{-2}\right) \\
          \displaystyle\prod_{i=1}^{p}d_{i}^{-2}\prod_{i=1}^{m}\left(1+d_{i}^{2}\right),
        \end{array}
      \right.
\end{equation}
\begin{equation}\label{sv2}
    \prod_{i=1}^{m}\left( \frac{d_{i}-d_{i}^{-1}}{d_{i}}\right)^{n+m-2p} =
            \left\{
                \begin{array}{l}
                    \displaystyle\prod_{i=1}^{m}\left(1 - d_{i}^{-2}\right)^{n+m-2p} \\
                    \displaystyle\prod_{i=1}^{m}
                    d_{i}^{-2(n+m-2p)}\prod_{i=1}^{m}\left(d_{i}^{2}-1\right)^{n+m-2p}.
                \end{array}
            \right.
\end{equation}
Also note that
\begin{eqnarray*}
  \left(d_{i}-d_{i}^{-1}\right)^{2}-\left(d_{j}-d_{j}^{-1}\right)^{2} &=& \left(\frac{d_{i}^{2}-1}{d_{i}}\right)^{2}
    - \left(\frac{d_{j}^{2}-1}{d_{j}}\right)^{2}\\
   &=& \frac{d_{j}^{2}\left(d_{i}^{2}-1\right)^{2} - d_{i}^{2}\left(d_{j}^{2}-1\right)^{2}}{d_{i}^{2}d_{j}^{2}} \\
   &=& \frac{d_{j}^{2}d_{i}^{4}-2d_{j}^{2}d_{i}^{2}+d_{j}^{2}-d_{i}^{2}d_{j}^{4}+2d_{i}^{2}d_{j}^{2}-d_{i}^{2}}{d_{i}^{2}d_{j}^{2}} \\
   &=& \frac{\left(d_{i}^{2}d_{j}^{2} - 1\right)\left(d_{i}^{2}-d_{j}^{2}\right)}{d_{i}^{2}d_{j}^{2}}.
\end{eqnarray*}
Thus
\begin{eqnarray}
  \prod_{i<j}^{p} \frac{\left(d_{i}-d_{i}^{-1}\right)^{2}-\left(d_{j}-d_{j}^{-1}\right)^{2}}{d_{i}^{2}-d_{j}^{2}}
  &=& \prod_{i<j}^{p} \frac{\displaystyle\frac{\left(d_{i}^{2}d_{j}^{2} - 1\right)\left(d_{i}^{2}-d_{j}^{2}\right)}
  {d_{i}^{2}d_{j}^{2}}}{d_{i}^{2}-d_{j}^{2}}\nonumber \\  \label{sv3}
   &=& \left\{
          \begin{array}{ll}
            \displaystyle\prod_{i=1}^{p} d_{i}^{-2(p-1)}\prod_{i<j}^{p}\left(d_{i}^{2}d_{j}^{2} - 1\right)\\
            \displaystyle\prod_{i<j}^{p}\left(1-d_{i}^{-2}d_{j}^{-2}\right).
          \end{array}
        \right.
\end{eqnarray}
Where we have used the expression
$$
  \prod_{i<j}^{p}\frac{1}{d_{i}^{2}d_{j}^{2}} = \prod_{i=1}^{p}d_{i}^{-2(p-1)}.
$$
\end{proof}
Substituting (\ref{sv1}), (\ref{sv2}) and (\ref{sv3}) into (\ref{sv0}) the desired results (\ref{dUU1})
are obtained.

\begin{thm}\label{teo1}
Consider the follow matrix transformation
\begin{equation}\label{mvBS}
    \mathbf{Z}  = \left (\mathbf{V}\mathbf{\Delta}^{+} - \mathbf{V}^{'+}\mathbf{\Delta}\right)\mathbf{\Xi}^{+},
\end{equation}
where $\mathbf{Z}$ and $\mathbf{V} \in \mathcal{L}_{m,n}(p)$, $\mathbf{\Delta}$ and $\mathbf{\Xi} \in
\mathcal{S}_{m}^{+}(s)$, $s \leq m$. Then
\begin{equation}\label{jz}
  (d\mathbf{Z}) = \frac{\G(\theta^{2})}{\displaystyle\prod_{i = 1}^{s}\ch_{i}(\mathbf{\Xi})^{n}\prod_{j = 1}^{s}\ch_{j}
  (\boldgreek{\beta})^{n/2}}(d\mathbf{V})
\end{equation}
with
$$
  \G(q,\theta^{2}) =
         \left\{
              \begin{array}{l}
                 \displaystyle\prod_{i=1}^{q}\left(1 - \theta_{i}^{-2}\right)^{n+m-2q} \left(1+\theta_{i}^{-2}\right)
                 \prod_{i<j}^{q}\left(1- \theta_{i}^{-2}\theta_{j}^{-2}\right)\\
                 \displaystyle\prod_{i=1}^{q}\theta_{i}^{-2(n+m-q)}\left(\theta_{i}^{2}-1\right)^{n+m-2q} \left(1+\theta_{i}^{2}\right)
                 \prod_{i<j}^{q}\left(\theta_{i}^{2}\theta_{j}^{2}-1\right).
              \end{array}
        \right.
$$
Here $\theta_{i}^{2} = \ch_{i}(\mathbf{V}'\mathbf{V}\boldgreek{\beta}^{+})$, $i = 1,\dots,q$, where $q$
denotes the rank of $\mathbf{V}\boldgreek{\beta}^{+}$, $q \leq \min(p,s)$ and $\boldgreek{\beta} =
\mathbf{\Delta}^{2}$.
\end{thm}
\begin{proof}
Denote
$$
    \mathbf{Z}  = \left (\mathbf{V}\mathbf{\Delta}^{+} -
 \mathbf{V}^{'+}\mathbf{\Delta}\right)\mathbf{\Xi}^{+} = \mathbf{U}\mathbf{\Xi}^{+},
$$
with $\mathbf{U} = \mathbf{Y}-\mathbf{Y}^{'+}$ and $\mathbf{Y} = \mathbf{V\Delta}^{+}$. Then by
\citet{dg:07}
\begin{equation}\label{jzu}
    (d\mathbf{Z}) = \prod_{i=1}^{s}\ch_{i}(\mathbf{\Xi}^{+}\mathbf{\Xi}^{+})^{n/2}(d\mathbf{U}) =
  \frac{(d\mathbf{U})}{\displaystyle\prod_{i=1}^{s}\ch_{i}(\mathbf{\Xi})^{n}}
\end{equation}
similarly, consider $\boldgreek{\beta}= \mathbf{\Delta}^{2}$, then
\begin{equation}\label{du}
    (d\mathbf{Y}) = \prod_{i=1}^{s}\ch_{i}(\mathbf{\Delta}^{+}\mathbf{\Delta}^{+})^{n/2}(d\mathbf{V}) =
      \prod_{i=1}^{s}\ch_{i}((\mathbf{\Delta}^{2})^{+})^{n/2}(d\mathbf{V}) =
     \frac{(d\mathbf{V})}{\displaystyle\prod_{i=1}^{s}\ch_{i}(\boldgreek{\beta})^{n/2}}.
\end{equation}
Now, substituting (\ref{dUU1}) and (\ref{du}) in (\ref{jzu}), the desired result is obtained, noting that
$\theta_{i}^{2} = \ch_{i}(\mathbf{Y}'\mathbf{Y}) =
\ch_{i}(\mathbf{\Delta}^{+}\mathbf{V}'\mathbf{V}\mathbf{\Delta}^{+})=
\ch_{i}(\mathbf{V}'\mathbf{V}(\mathbf{\Delta}^{+})^{2})=\ch_{i}(\mathbf{V}'\mathbf{V}\boldgreek{\beta}^{+})$
$i = 1,\dots,q$, where $q$ is the rank of $\mathbf{V}\boldgreek{\beta}^{+}$, $q \leq \min(p,s)$.
\end{proof}

\section{Singular matrix variate generalised Birnbaum-Saunders distribution}\label{sec:3}

As a first goal, the density of the matrix variate $\mathbf{V}$ defined in (\ref{mvBS}) is obtained when
$\mathbf{Z}$ has a singular matrix variate elliptically contoured distribution. This distribution shall be
termed \emph{singular matrix variate generalised square root Birnbaum-Saunders distribution}. Then the
main result is derived by finding the density of the \emph{singular matrix variate generalised Birnbaum-Saunders
distribution}. Finally, some special cases and basic properties of the singular
matrix variate generalised Birnbaum-Saunders distribution can be derived.

Now,  instead of making the change of variable (\ref{nem}), where $Z \sim
\mathcal{E}(0,1;h)$, we can propose the transformation
$$
  Y = \alpha Z = \frac{V}{\sqrt{\beta}} - \frac{\sqrt{\beta}}{V}
$$
where  $Y \sim \mathcal{E}(0,\alpha^{2};h)$. Next, we extend this idea to the singular matrix variate
case.

\begin{thm}\label{teo2}
Assume that $\mathbf{Y} \sim \mathcal{E}_{n \times m}^{n,s}(\mathbf{0}_{n \times m}, \mathbf{I}_{n}
\otimes \mathbf{\Xi}^{2}, h)$ and consider the following matrix transformation
\begin{equation}\label{mnem}
    \mathbf{Y}(=\mathbf{Z}\mathbf{\Xi})= \left (\mathbf{V}\mathbf{\Delta}^{+} - \mathbf{V}^{'+}\mathbf{\Delta}\right
    ),
\end{equation}
where $\mathbf{\Xi}^{2} \in \mathcal{S}_{m}^{+}(s)$ is the shape parameter matrix; $\mathbf{\Delta} \in
\mathcal{S}_{m}^{+}(s)$, is the scale parameter matrix, such that $\mathbf{\Delta}^{2} =
\boldgreek{\beta}$; and $\mathbf{V} \in \mathcal{L}_{m,n}^{+}(p)$, $p = \min(n,s)$. Then the density with
respect to Hausdorff measure is
\begin{eqnarray*}
  dF_{\mathbf{V}}(\mathbf{V}) &=& \frac{\G(q,\theta^{2})}{\displaystyle\prod_{i = 1}^{s}\ch_{i}(\mathbf{\Xi})^{n}\prod_{j = 1}^{s}\ch_{j}
  (\boldgreek{\beta})^{n/2}} \\
   && \times h\left[\tr \mathbf{\Xi}^{+2}
   \left(\mathbf{\Delta}^{+}\mathbf{V}'\mathbf{V}\mathbf{\Delta}^{+} + \mathbf{\Delta}\left(\mathbf{V}'
   \mathbf{V}\right)^{+}\mathbf{\Delta} - 2 \mathbf{\Delta}\mathbf{V}^{+}
   \mathbf{V}\mathbf{\Delta}^{+}\right) \right](d\mathbf{V}),
\end{eqnarray*}
where $(d\mathbf{V})$ is defined in (\ref{h1}), $\mathbf{\Xi}^{+2} = (\mathbf{\Xi}^{2})^{+}$ and
$$
  \G(q,\theta^{2}) =
         \left\{
              \begin{array}{l}
                 \displaystyle\prod_{i=1}^{q}\left(1 - \theta_{i}^{-2}\right)^{n+m-2q} \left(1+\theta_{i}^{-2}\right)
                 \prod_{i<j}^{q}\left(1- \theta_{i}^{-2}\theta_{j}^{-2}\right)\\
                 \displaystyle\prod_{i=1}^{q}\theta_{i}^{-2(n+m-q)}\left(\theta_{i}^{2}-1\right)^{n+m-2q} \left(1+\theta_{i}^{2}\right)
                 \prod_{i<j}^{q}\left(\theta_{i}^{2}\theta_{j}^{2}-1\right),
              \end{array}
        \right.
$$
$\theta_{i}^{2} = \ch_{i}(\mathbf{V}'\mathbf{V}\boldgreek{\beta}^{+})$, $i = 1,\dots,q$, where $q$
denotes de rank of $\mathbf{V\Delta}^{+}$, $q \leq \min(p,s)$.
\end{thm}
\begin{proof}
By Definition \ref{def1}, the density of $\mathbf{Y}$ is
$$
  dF_{_{\mathbf{Y}}}(\mathbf{Y}) = \frac{1}{\displaystyle\prod_{i = 1}^{s}\ch_{i}(\mathbf{\Xi})^{n}} h(\tr
  \mathbf{\Xi}^{+2} \mathbf{Y}'\mathbf{Y}) (d\mathbf{Y}).
$$
Now, making the change of variable (\ref{mnem}), with corresponding Jacobian established in Theorem \ref{teo1}, and
observing that
\begin{eqnarray*}
  \tr \mathbf{\Xi}^{+2} \mathbf{Y}'\mathbf{Y} &=& \tr \mathbf{\Xi}^{+2} \left (\mathbf{V}\mathbf{\Delta}^{+}
  - \mathbf{V}^{'+}\mathbf{\Delta}\right)' \left (\mathbf{V}\mathbf{\Delta}^{+} - \mathbf{V}^{'+} \mathbf{\Delta}
  \right)\\
   &=& \tr \mathbf{\Xi}^{+2}
   \left(\mathbf{\Delta}^{+}\mathbf{V}'\mathbf{V}\mathbf{\Delta}^{+} + \mathbf{\Delta}\left(\mathbf{V}'
   \mathbf{V}\right)^{+}\mathbf{\Delta} - 2 \mathbf{\Delta}\mathbf{V}^{+}
   \mathbf{V}\mathbf{\Delta}^{+}\right).
\end{eqnarray*}
The density of $\mathbf{V}$ is obtained.
\end{proof}

The density of a singular matrix variate generalised Birnbaum-Saunders distribution is obtained in the
following result. This fact shall be denoted as $\mathbf{T} \sim
\mathcal{GBS}_{m}^{p}(q,s,n,\mathbf{\Xi}, \boldgreek{\beta},h)$, where $\mathbf{\Xi} \in
\mathcal{S}_{m}^{+}(s)$, is the shape parameter matrix, $\mathbf{\Delta} \in \mathcal{S}_{m}^{+}(s)$,
such that $\boldgreek{\beta}$ is the scale parameter matrix where $\mathbf{\Delta}^{2} =
\boldgreek{\beta}$. Also, note that $q$ denotes the rank of $\mathbf{V\Delta}^{+}$, $q \leq \min(p,s)$,
and $p = \min(n,s)$.

\begin{thm}\label{teo3}
Suppose that $\mathbf{T} \sim \mathcal{GBS}_{m}^{p}(q,s,n,\mathbf{\Xi}, \boldgreek{\beta}, h)$,
$\mathbf{T} \in \mathcal{S}_{m}^{+}(p)$, $\mathbf{\Xi} \in \mathcal{S}_{m}^{+}(s)$, and
$\boldgreek{\beta} \in \mathcal{S}_{m}^{+}(s)$; $\boldgreek{\beta} = (\mathbf{\Delta})^{2}$ and $p =
\min(n,s)$. Then the density of $\mathbf{T}$ with respect to the Hausdorf measure is
$$
    dF_{\mathbf{T}}(\mathbf{T})= \frac{\pi^{np/2}\G(q,\delta)}{2^{p}\Gamma_{p}[n/2]
    \displaystyle\prod_{i = 1}^{s}\ch_{i}(\mathbf{\Xi})^{n}\prod_{j = 1}^{s}\ch_{j}
  (\boldgreek{\beta})^{n/2}} |\mathbf{\Lambda}|^{(n-m-1)/2}\hspace{2cm}
$$
$$
    \times
    h\left[\tr \mathbf{\Xi}^{+2}\left(\mathbf{\Delta}^{+}\mathbf{T}\mathbf{\Delta}^{+} +
    \mathbf{\Delta}\mathbf{T}^{+}\mathbf{\Delta} - 2 \mathbf{\Delta T}^{+}\mathbf{T\Delta}^{+}\right) \right](d\mathbf{T}),
$$
where if $\mathbf{T} = \mathbf{V}'\mathbf{V} = \mathbf{Q}_{1}\mathbf{\Lambda Q}'_{1}$, $\mathbf{Q}_{1}
\in \mathcal{V}_{p,m}$ and $\mathbf{\Lambda} = \diag(\lambda_{1}, \dots,\lambda_{p})$, $\lambda_{1}>
\cdots >\lambda_{p}> 0$, then \citep{dggj:97}
\begin{equation}\label{h2}
    (d\mathbf{T}) = 2^{-p} \prod_{i=1}^{p}\lambda_{i}^{m - p} \prod_{i < j}^{p}(\lambda_{i} - \lambda_{j})
          (\mathbf{Q}'_{1}d\mathbf{Q}_{1})\bigwedge_{i=1}^{p}d\lambda_{i},
\end{equation}
and
$$
  \G(q,\delta) =
         \left\{
              \begin{array}{l}
                 \displaystyle\prod_{i=1}^{q}\left(1 - \delta_{i}^{-1}\right)^{n+m-2q} \left(1+\delta_{i}^{-1}\right)
                 \prod_{i<j}^{q}\left(1- \delta_{i}^{-1}\delta_{j}^{-1}\right)\\
                 \displaystyle\prod_{i=1}^{q}\delta_{i}^{-(n+m-q)}\left(\delta_{i}-1\right)^{n+m-2q} \left(1+\delta_{i}\right)
                 \prod_{i<j}^{q}\left(\delta_{i}\delta_{j}-1\right).
              \end{array}
        \right.
$$
Here $\delta_{i} = \ch_{i}(\boldgreek{\beta}^{+}\mathbf{T})$, $i = 1,\dots,q$, $q$ is the rank of $
\boldgreek{\beta}^{+}\mathbf{T}$, $q \leq \min(p,s)$, and $\Gamma_{p}[\cdot]$ denotes de multivariate
gamma function, see \citet[Definition 2.1.10, p.61]{mh:05},
$$
  \Gamma_{p}[a] = \pi^{p(p-1)/4} \prod_{i=1}^{p} \Gamma[a-(i-1)/2], [\re(a)>(p-1)/2]
$$
where $\re(\cdot)$ denotes de real part of the argument.
\end{thm}
\begin{proof}
The density of $\mathbf{V}$ is
\begin{eqnarray*}
   dF_{\mathbf{V}}(\mathbf{V}) &=& \frac{\G(q,\theta^{2})}{\displaystyle\prod_{i = 1}^{s}\ch_{i}(\mathbf{\Xi})^{n}
   \prod_{j = 1}^{s}\ch_{j}(\boldgreek{\beta})^{n/2}} \\
   &\times&  h\left[\tr \mathbf{\Xi}^{+2} \left(\mathbf{\Delta}^{+}\mathbf{V}'\mathbf{V}\mathbf{\Delta}^{+}
   + \mathbf{\Delta}\left(\mathbf{V}' \mathbf{V}\right)^{+}\mathbf{\Delta} - 2 \mathbf{\Delta}\mathbf{V}^{+}
   \mathbf{V}\mathbf{\Delta}^{+}\right) \right](d\mathbf{V})
\end{eqnarray*}
where
$$
  \G(q,\theta^{2}) =
         \left\{
              \begin{array}{l}
                 \displaystyle\prod_{i=1}^{q}\left(1 - \theta_{i}^{-2}\right)^{n+m-2q} \left(1+\theta_{i}^{-2}\right)
                 \prod_{i<j}^{q}\left(1- \theta_{i}^{-2}\theta_{j}^{-2}\right)\\
                 \displaystyle\prod_{i=1}^{q}\theta_{i}^{-2(n+m-q)}\left(\theta_{i}^{2}-1\right)^{n+m-2q} \left(1+\theta_{i}^{2}\right)
                 \prod_{i<j}^{q}\left(\theta_{i}^{2}\theta_{j}^{2}-1\right),
              \end{array}
        \right.
$$
$\theta_{i}^{2} = \ch_{i}(\mathbf{V}'\mathbf{V}\boldgreek{\beta}^{+})$, $i = 1,\dots,q$, and $q$
is the rank of $\mathbf{V\Delta}^{+}$, $q \leq \min(p,s)$.

Define $\mathbf{T} = \mathbf{V}'\mathbf{V}$ with $\mathbf{V} \in \mathcal{L}_{m,n}^{+}(p)$, $\mathbf{V} =
\mathbf{H}_{1}\mathbf{\Lambda}^{1/2}\mathbf{Q}'_{1}$, where $\mathbf{H}_{1} \in \mathcal{V}_{p,n}$ and
$\mathbf{Q}_{1} \in \mathcal{V}_{p,m}$ and $\mathbf{\Lambda}^{1/2} = \diag(\lambda_{1}^{1/2},
\dots,\lambda_{p}^{1/2})$, $\lambda_{1}^{1/2}> \cdots
>\lambda_{p}^{1/2}> 0$. Then $\mathbf{T} = \mathbf{V}'\mathbf{V} = \mathbf{Q}_{1}\mathbf{\Lambda
Q}'_{1}$. Note that in the singular value factorisation under consideration, $\mathbf{V} =
\mathbf{H}_{1}\mathbf{\Lambda}^{1/2}\mathbf{Q}'_{1}$, the matrices $\mathbf{H}_{1}$ and $\mathbf{Q}_{1}$
are defined in \citet[p. 115]{m:97}, see Theorem 2.12. Then by \citet{dggg:05a}
$$
  (d\mathbf{V}) = 2^{-p} |\mathbf{\Lambda}|^{(n-m-1)/2}(d\mathbf{T})(\mathbf{H}'_{1}d\mathbf{H}_{1}).
$$
In particular, from \citet[p. 117]{m:97},
\begin{equation}\label{stiefel}
    \int_{\mathbf{H}_{1} \in \mathcal{V}_{p,n}} (\mathbf{H}'_{1}d\mathbf{H}_{1}) =
  \frac{\pi^{pn/2}}{\Gamma_{p}[n/2]}.
\end{equation}
 In addition note that
\begin{eqnarray*}
  \mathbf{T}^{+}\mathbf{T} &=& (\mathbf{V}'\mathbf{V})^{+}(\mathbf{V}'\mathbf{V}) = \mathbf{V}^{+}
  \mathbf{V}^{'+} \mathbf{V}'\mathbf{V}\\
  &=& \mathbf{V}^{+}(\mathbf{VV}^{+})'\mathbf{V} = \mathbf{V}^{+}\mathbf{VV}^{+}\mathbf{V}\\
   &=& \mathbf{V}^{+}\mathbf{V}.
\end{eqnarray*}

Hence, the joint density function $dF_{\mathbf{T},\mathbf{H}_{1}}(\mathbf{T},\mathbf{H}_{1})$ of
$\mathbf{T}$ and $\mathbf{H}_{1}$ is
$$
  = \frac{\G(q,\delta)}{2^{p}|
  \displaystyle\prod_{i = 1}^{s}\ch_{i}(\mathbf{\Xi})^{n}\prod_{j = 1}^{s}\ch_{j}
  (\boldgreek{\beta})^{n/2}} |\mathbf{\Lambda}|^{(n-m-1)/2} \hspace{7cm}
$$
$$
  \hspace{2cm} \times
  h\left[\tr \mathbf{\Xi}^{+2}\left(\mathbf{\Delta}^{+}\mathbf{T}\mathbf{\Delta}^{+} + \mathbf{\Delta}\mathbf{T}^{+}
  \mathbf{\Delta} - 2 \mathbf{\Delta T}^{+}\mathbf{T\Delta}^{+}\right) \right](d\mathbf{T})(\mathbf{H}'_{1}d\mathbf{H}_{1}),
$$
where
$$
  \G(q,\delta) =
         \left\{
              \begin{array}{l}
                 \displaystyle\prod_{i=1}^{q}\left(1 - \delta_{i}^{-1}\right)^{n+m-2q} \left(1+\delta_{i}^{-1}\right)
                 \prod_{i<j}^{q}\left(1- \delta_{i}^{-1}\delta_{j}^{-1}\right)\\
                 \displaystyle\prod_{i=1}^{q}\delta_{i}^{-(n+m-q)}\left(\delta_{i}-1\right)^{n+m-2q} \left(1+\delta_{i}\right)
                 \prod_{i<j}^{q}\left(\delta_{i}\delta_{j}-1\right).
              \end{array}
        \right.
$$
where $\delta_{i} = \ch_{i}(\boldgreek{\beta}^{+}\mathbf{T})$, $i = 1,\dots,q$, $q=$ rank of $
\boldgreek{\beta}^{+}\mathbf{T}$, $q \leq \min(p,s)$.

Integration with respect to $\mathbf{H}_{1}$, by using (\ref{stiefel}), provides the stated marginal
density function for $\mathbf{T}$.
\end{proof}

Three cases of particular interest are given next:
\begin{enumerate}
  \item Nonsingular case. In this case $q = p = s = m$.
  \item The classic case, which is obtained assuming that $\mathbf{Y}$ as a singular matrix variate
    normal distribution in Theorem \ref{teo1}. Therefore, from Theorem \ref{teo3} we obtain the singular
    matrix variate Birnbaum-Saunders distribution, which shall be denoted as $\mathbf{T} \sim \mathcal{BS}_{m}^{p}(q,s,n,
    \mathbf{\Xi}, \boldgreek{\beta})$.
  \item When $\boldgreek{\beta} = \beta \mathbf{I}_{m}$, $\beta > 0$, i.e. when $\mathbf{T} \sim \mathcal{GBS}_{m}^{p}(q,s,n,
  \mathbf{\Xi}, \beta \mathbf{I}_{m},h)$.
\end{enumerate}
The results are summarised in following three Corollaries, respectively.

\begin{cor}
Let $\mathbf{T} \sim \mathcal{GBS}_{m}(n,\mathbf{\Xi}, \boldgreek{\beta}, h)$, $\mathbf{T} \in
\mathcal{S}_{m}$, $\mathbf{\Xi} \in \mathcal{S}_{m}$, and $\boldgreek{\beta} \in \mathcal{S}_{m}$;
$\boldgreek{\beta} = (\mathbf{\Delta})^{2}$. Then the density of $\mathbf{T}$ with respect to the
Lebesgue measure on $\mathcal{S}_{m}$ is
$$
    dF_{\mathbf{T}}(\mathbf{T})= \frac{\pi^{nm/2}\G(m,\delta)}{2^{p}\Gamma_{m}[n/2]
    |\mathbf{\Xi}|^{n}|\boldgreek{\beta}|^{n/2}} |\mathbf{T}|^{(n-m-1)/2}\hspace{5cm}
$$
$$
    \hspace{2cm}\times
    h\left[\tr \mathbf{\Xi}^{-2}\left(\mathbf{\Delta}^{-1}\mathbf{T}\mathbf{\Delta}^{-1} +
    \mathbf{\Delta}\mathbf{T}^{-1}\mathbf{\Delta} - 2 \mathbf{I}_{m}\right) \right](d\mathbf{T}),
$$
where
$$
  \G(m,\delta) =
         \left\{
              \begin{array}{l}
                 \displaystyle\prod_{i=1}^{m}\left(1 - \delta_{i}^{-1}\right)^{n-m} \left(1+\delta_{i}^{-1}\right)
                 \prod_{i<j}^{m}\left(1- \delta_{i}^{-1}\delta_{j}^{-1}\right)\\
                 \displaystyle\prod_{i=1}^{m}\delta_{i}^{-n}\left(\delta_{i}-1\right)^{n-m} \left(1+\delta_{i}\right)
                 \prod_{i<j}^{m}\left(\delta_{i}\delta_{j}-1\right),
              \end{array}
        \right.
$$
and $\delta_{i} = \ch_{i}(\boldgreek{\beta}^{-1}\mathbf{T})$, $i = 1,\dots,m$.
\end{cor}
\begin{proof}
The result follows by taking $q = p = s = m$ in Theorem \ref{teo3} and noting that $2 \mathbf{\Delta
T}^{+}\mathbf{T\Delta}^{+} = 2 \mathbf{\Delta T}^{-1}\mathbf{T\Delta}^{-1} = 2 \mathbf{I}_{m}$.
\end{proof}

\begin{cor}\label{cor1}
Suppose that $\mathbf{T} \sim \mathcal{BS}_{m}^{p}(q,s,n,\mathbf{\Xi}, \boldgreek{\beta})$, $\mathbf{T}
\in \mathcal{S}_{m}^{+}(p)$, $\mathbf{\Xi} \in \mathcal{S}_{m}^{+}(s)$, and $\boldgreek{\beta} \in
\mathcal{S}_{m}^{+}(s)$; $\boldgreek{\beta} = (\mathbf{\Delta})^{2}$ and $p = \min(n,s)$. Then the
density of $\mathbf{T}$ with respect to the Hausdorf measure is
$$
    dF_{\mathbf{T}}(\mathbf{T})= \frac{\pi^{n(p-s)/2}\G(q,\delta)}{2^{p+ns/2}\Gamma_{p}[n/2]
    \displaystyle\prod_{i = 1}^{s}\ch_{i}(\mathbf{\Xi})^{n}\prod_{j = 1}^{s}\ch_{j}
    (\boldgreek{\beta})^{n/2}} |\mathbf{\Lambda}|^{(n-m-1)/2}\hspace{3cm}
$$
$$
    \hspace{3cm}\times
    \etr\left[-\frac{1}{2} \mathbf{\Xi}^{+2}\left(\mathbf{\Delta}^{+}\mathbf{T}\mathbf{\Delta}^{+} +
    \mathbf{\Delta}\mathbf{T}^{+}\mathbf{\Delta} - 2 \mathbf{\Delta T}^{+}\mathbf{T\Delta}^{+}\right) \right](d\mathbf{T}),
$$
where the element of volumen $(d\mathbf{T})$ is defined in (\ref{h2}),
$$
  \G(q,\delta) =
         \left\{
              \begin{array}{l}
                 \displaystyle\prod_{i=1}^{q}\left(1 - \delta_{i}^{-1}\right)^{n+m-2q} \left(1+\delta_{i}^{-1}\right)
                 \prod_{i<j}^{q}\left(1- \delta_{i}^{-1}\delta_{j}^{-1}\right)\\
                 \displaystyle\prod_{i=1}^{q}\delta_{i}^{-(n+m-q)}\left(\delta_{i}-1\right)^{n+m-2q} \left(1+\delta_{i}\right)
                 \prod_{i<j}^{q}\left(\delta_{i}\delta_{j}-1\right),
              \end{array}
        \right.
$$
and $\delta_{i} = \ch_{i}(\boldgreek{\beta}^{+}\mathbf{T})$, $i = 1,\dots,q$, $q$ denotes the rank of $
\boldgreek{\beta}^{+}\mathbf{T}$, $q \leq \min(p,s)$, and $\etr(\cdot)= \exp(\tr(\cdot))$.
\end{cor}
\begin{proof}
Defining  $h(z) =(2\pi)^{-ns/2} \etr(-z/2)$, the proof  follows straightforwardly from Theorem
\ref{teo3}.
\end{proof}

In following case, note that $\boldgreek{\beta} = \mathbf{\Delta}^{2} = \beta\mathbf{I}_{m}$, hence
$\mathbf{\Delta} = \sqrt{\beta}\mathbf{I}_{m}$.

\begin{cor}\label{cor2}
If $\mathbf{T} \sim \mathcal{GBS}_{m}^{p}(q,s,n,\mathbf{\Xi}, \beta\mathbf{I}_{m}, h)$, with $p =
\min(n,s)$, $\mathbf{T} \in \mathcal{S}_{m}^{+}(p)$, $\mathbf{\Xi} \in \mathcal{S}_{m}^{+}(s)$, and
$\beta > 0$. Then the density of $\mathbf{T}$ with respect to the Hausdorf measure is
$$
    dF_{\mathbf{T}}(\mathbf{T})= \frac{\pi^{np/2}\G(p,\delta)}{2^{p}\Gamma_{p}[n/2]\beta^{mn/2}
    \displaystyle\prod_{i = 1}^{s}\ch_{i}(\mathbf{\Xi})^{n}} |\mathbf{\Lambda}|^{(n-m-1)/2}\hspace{4cm}
$$
$$
    \hspace{3cm}\times
    h\left[\tr \mathbf{\Xi}^{+2}\left(\frac{1}{\beta}\mathbf{T} +
    \beta\mathbf{T}^{+} - 2 \mathbf{T}^{+}\mathbf{T}\right) \right](d\mathbf{T}),
$$
where the Hasdorff measure $(d\mathbf{T})$ is explicitly defined in (\ref{h2}) and
$$
  \G(p,\rho) =
         \left\{
              \begin{array}{l}
                 \displaystyle\prod_{i=1}^{p}\left(1 - \beta\rho_{i}^{-1}\right)^{n+m-2p} \left(1+\beta\rho_{i}^{-1}\right)
                 \prod_{i<j}^{p}\left(1- \beta^{^{2}}\rho_{i}^{-1}\rho_{j}^{-1}\right)\\
                 \beta^{p(n+m-p)}\displaystyle\prod_{i=1}^{p}\rho_{i}^{-(n+m-p)}\left(\frac{\rho_{i}}
                 {\beta}-1\right)^{n+m-2p} \left(1+\frac{\rho_{i}}{\beta}\right)
                 \prod_{i<j}^{p}\left(\frac{\rho_{i}\rho_{j}}{\beta^{2}}-1\right)
              \end{array}
        \right.
$$
with $\rho_{i} = \ch_{i}(\mathbf{T})$, $i = 1,\dots,p$.
\end{cor}
\begin{proof}
The result follows straightforwardly from Theorem \ref{teo2}.
\end{proof}

Now, two basic properties of the singular matrix variate generalised Birnbaum-Saunders distribution are
summarised in the follow result.

\begin{thm}\label{teo4}
Assume that $\mathbf{T} \sim \mathcal{GBS}_{m}^{p}(q,s,n,\mathbf{\Xi}, \boldgreek{\beta},h)$, then
\begin{description}
  \item[i)] Consider the singular matrix transformation $\mathbf{S} = \mathbf{T}^{+}$, $\mathbf{T},
  \mathbf{S} \in \mathcal{S}_{m}^{+}(p)$. Then the density
  of $\mathbf{S}$ with respect to Hausdorff is given by
  $$
    dF_{\mathbf{S}}(\mathbf{S})= \frac{\pi^{np/2}\G(q,\psi)}{2^{p}\Gamma_{p}[n/2]
    \displaystyle\prod_{i = 1}^{s}\ch_{i}(\mathbf{\Xi})^{n}\prod_{j = 1}^{s}\ch_{j}
  (\boldgreek{\beta})^{n/2}} |\mathbf{\Omega}|^{-(n+3m+1)/2+p}\hspace{2cm}
$$
$$
    \times
    h\left[\tr \mathbf{\Xi}^{+2}\left(\mathbf{\Delta}^{+}\mathbf{S}^{+}\mathbf{\Delta}^{+} +
    \mathbf{\Delta}\mathbf{S}\mathbf{\Delta} - 2 \mathbf{\Delta S}\mathbf{S}^{+}\mathbf{\Delta}^{+}\right) \right](d\mathbf{S}),
$$
where the measure $(d\mathbf{S})$ is given by (\ref{h2}), $\mathbf{S} =
\mathbf{Q}_{1}\mathbf{\Omega}\mathbf{Q}'_{1}$, $\mathbf{Q}_{1} \in \mathcal{V}_{p,m}$, $\mathbf{\Omega} =
\diag(\omega_{1}, \dots, \omega_{p})$, $\omega_{1}> \cdots > \omega_{p}>0$,
$$
  \G(q,\psi) =
         \left\{
              \begin{array}{l}
                 \displaystyle\prod_{i=1}^{q}\left(1 - \psi_{i}^{-1}\right)^{n+m-2q} \left(1+\psi_{i}^{-1}\right)
                 \prod_{i<j}^{q}\left(1- \psi_{i}^{-1}\psi_{j}^{-1}\right)\\
                 \displaystyle\prod_{i=1}^{q}\psi_{i}^{-(n+m-q)}\left(\psi_{i}-1\right)^{n+m-2q} \left(1+\psi_{i}\right)
                 \prod_{i<j}^{q}\left(\psi_{i}\psi_{j}-1\right),
              \end{array}
        \right.
$$
and $\psi_{i} = \ch_{i}(\boldgreek{\beta}^{+}\mathbf{S}^{+})$, $i = 1,\dots,q$, $q$ denotes the rank of $
\boldgreek{\beta}^{+}\mathbf{S}^{+}$, $q \leq \min(p,s)$.
  \item[ii)] Let $\mathbf{Y}, \mathbf{T} \in \mathcal{S}_{m}^{+}(n)$ and define $\mathbf{Y} = \mathbf{C}'\mathbf{TC}$,
  $\mathbf{C} \in \mathcal{L}_{m,m}(m)$. Then the density of $\mathbf{Y}$ with respect to Hausdorff measure is
  $$
    dF_{\mathbf{T}}(\mathbf{T})= \frac{\pi^{n^{2}/2}\G(q,\eta)}{2^{n}\Gamma_{n}[n/2]
    \displaystyle\prod_{i = 1}^{s}\ch_{i}(\mathbf{\Xi})^{n}\prod_{j = 1}^{s}\ch_{j}
  (\boldgreek{\beta})^{n/2} |\mathbf{C}|^{n}} |\mathbf{L}|^{(n-m-1)/2}\hspace{2cm}
$$
$$
    \times
    h\left[\tr \mathbf{\Xi}^{+2}\left((\mathbf{\Delta C})^{'+}\mathbf{Y}(\mathbf{\Delta C})^{+} +
    \mathbf{\Delta C}\mathbf{Y}^{+}(\mathbf{\Delta C})' - 2 \mathbf{\Delta C Y}^{+}\mathbf{Y}
    (\mathbf{\Delta C})^{+}\right) \right](d\mathbf{Y}),
$$
where the element of volume $(d\mathbf{Y})$ is defined in (\ref{h2}), $\mathbf{Y} =
\mathbf{P}_{1}\mathbf{L}\mathbf{P}'_{1}$, $\mathbf{P}_{1} \in \mathcal{V}_{n,m}$, $\mathbf{L} =
\diag(l_{1}, \dots, l_{n})$, $l_{1}> \cdots > l_{n}>0$,
$$
  \G(q,\eta) =
         \left\{
              \begin{array}{l}
                 \displaystyle\prod_{i=1}^{q}\left(1 - \eta_{i}^{-1}\right)^{n+m-2q} \left(1+\eta_{i}^{-1}\right)
                 \prod_{i<j}^{q}\left(1- \eta_{i}^{-1}\eta_{j}^{-1}\right)\\
                 \displaystyle\prod_{i=1}^{q}\eta_{i}^{-(n+m-q)}\left(\eta_{i}-1\right)^{n+m-2q} \left(1+\eta_{i}\right)
                 \prod_{i<j}^{q}\left(\eta_{i}\eta_{j}-1\right).
              \end{array}
        \right.
$$
Here $q$ denotes the rank of $(\mathbf{C}\boldgreek{\beta}\mathbf{C}')^{+}\mathbf{Y}$, $\eta_{i} =
\ch_{i}((\mathbf{C}\boldgreek{\beta}\mathbf{C}')^{+}\mathbf{Y})$, $i = 1,\dots,q$, $q \leq \min(n,s)$.
\end{description}
\end{thm}
\begin{proof}
The corresponding proofs are obtained by considering the following Jacobians, see \citet{dggj:06a} and
\citet{dggj:97}, respectively.
\begin{description}
  \item[i)] Let $\mathbf{S} = \mathbf{T}^{+}$, then $(d\mathbf{T}) = |\mathbf{\Omega}|^{-2m+p-1}
  (d\mathbf{S})$, where $\mathbf{S} = \mathbf{Q}_{1}\mathbf{\Omega}\mathbf{Q}'_{1}$, $\mathbf{Q}_{1} \in
  \mathcal{V}_{p,m}$, $\mathbf{\Omega} = \diag(\omega_{1}, \dots, \omega_{p})$, $\omega_{1}> \cdots >
  \omega_{p}>0$ and
  \item[ii)] Define $\mathbf{Y} = \mathbf{C}'\mathbf{TC}$, then
  $$
    (d\mathbf{T}) = |\mathbf{\Lambda}|^{(m+1-n)/2}|\mathbf{L}|^{-(m+1-n)/2} |\mathbf{C}|^{-n}(d\mathbf{Y}),
  $$
  where $\mathbf{T} = \mathbf{Q}_{1}\mathbf{\Lambda}\mathbf{Q}'_{1}$, $\mathbf{Q}_{1} \in
  \mathcal{V}_{n,m}$, $\mathbf{\Lambda} = \diag(\lambda_{1}, \dots, \lambda_{n})$, $\lambda_{1}> \cdots >
  \lambda_{n}>0$ and $\mathbf{Y} = \mathbf{P}_{1}\mathbf{L}\mathbf{P}'_{1}$, $\mathbf{P}_{1} \in
  \mathcal{V}_{n,m}$, $\mathbf{L} = \diag(l_{1}, \dots, l_{n})$, $l_{1}> \cdots > l_{n}>0$.
\end{description}
respectively.
\end{proof}

Finally, the reader can apply this work by using particular kernels of the most common elliptical
distributions, such as Kotz, Pearson VII and II type distributions. The distributions just follow by
changing $h(\cdot)$ with the corresponding model in a similar way of the proof given for Corollary
\ref{cor1}.

\section{Conclusions}
A version of the singular matrix variate generalised Birnbaum-Saunders distribution  with respect to
Hausdorff measure has been proposed in this article. Several basic properties and particular cases of
this distribution were also studied. The work has contributed to the unification of the matrix variate
Birnbaum-Saunders in a general setting that includes families of singular and non singular elliptical
contoured distributions. During the analysis, the associated Jacobian theory has been established and
solving the problems of conceptualisation of the literature about different pairs of densities and
measures. Finally, the work is easily extended to real normed division algebras, which is valid for
complex, quaternions and octonions.



\end{document}